\title{Suppression of chemotactic singularity by buoyancy}
\author{Zhongtian Hu\thanks{
                  Department of
Mathematics, Duke University, Durham, NC, 27708, USA; email: zhongtian.hu@duke.edu}
\and Alexander Kiselev
\thanks{Department of
Mathematics, Duke University, Durham, NC, 27708, USA; email: kiselev@math.duke.edu}
\and Yao Yao
\thanks{Department of Mathematics, National University of Singapore, 119076 Singapore; email: yaoyao@nus.edu.sg}
        }
\documentclass[12pt]{article}
\usepackage[total={6.5in, 8.5in}]{geometry}%\usepackage{showkeys}

\usepackage{comment,slashed,tensor}
\usepackage[T1]{fontenc}
\usepackage{ stmaryrd }

 \usepackage[bookmarksnumbered, bookmarksopen, colorlinks, citecolor=red, linkcolor=black]{hyperref}
 \usepackage{amsfonts}
\usepackage{graphicx}
 \usepackage{amsmath, amssymb}
 \usepackage{amsthm}
 \usepackage{mathtools}
\usepackage[noabbrev]{cleveref}

\numberwithin{equation}{section}

\usepackage{xcolor}
\definecolor{purple}{rgb}{0.5, 0, 1}
\definecolor{orange}{rgb}{1,.5,0}

\newtheorem{thm}{Theorem}[section]
\newtheorem{lem}[thm]{Lemma}
\newtheorem{prop}[thm]{Proposition}
\newtheorem{rmk}[thm]{Remark}
\newtheorem{cor}[thm]{Corollary}

\newtheorem{definition}[thm]{Definition}

\newcommand{\R}{\mathbb{R}}
\newcommand{\N}{\mathbb{N}}
\newcommand{\Z}{\mathbb{Z}}

\newcommand{\T}{\mathbb{T}}

\newcommand{\mC}{\mathcal{C}}

\newcommand{\brho}{\bar\rho}
\newcommand{\trho}{\tilde{\rho}}
\newcommand{\barf}{\bar f}
\newcommand{\tf}{\tilde{f}}

\newcommand{\aaa}{a}

\DeclareMathOperator{\divv}{div}

\DeclareMathOperator*{\esssup}{ess\,sup}

\newcommand{\p}{\partial}

\begin{document}
\newpage

\maketitle
%%%%%%%%%%%%%%%%%%%%%%%%%%%
% abstract, keywords and Subject classification are optional.
%%%%%%%%%%%%%%%%%%%%%%%%%%%

\begin{abstract}
Chemotactic singularity formation in the context of the Patlak-Keller-Segel equation is an extensively studied phenomenon.
In recent years, it has been shown that the presence of fluid advection can arrest the singularity formation given that
the fluid flow possesses mixing or diffusion enhancing properties and its amplitude is sufficiently strong -
this effect is conjectured to hold for more general classes of nonlinear PDEs.
In this paper, we consider the Patlak-Keller-Segel equation coupled  with a fluid flow that obeys
Darcy's law for incompressible porous media via buoyancy force. We prove that in contrast with passive advection, this active fluid
coupling is capable of suppressing singularity formation at arbitrary small coupling strength: namely, the system
always has globally regular solutions.
\end{abstract}

% Most people don't use these, so they are "commented out"
% by starting the lines with a "%"
%\begin{keywords}
%   \LaTeX, typesetting
%\end{keywords}

%\begin{AMS}
%   50C60, 18C25
%\end{AMS}

%%%%%%%%%%%%%%%%%%%%%%
% % Here is the start of the Text
%%%%%%%%%%%%%%%%%%%%%%
\section{Introduction}

There are many mechanisms for regularity in partial differential equations. Sometimes, like in the 2D Euler equation,
there is a controlled quantity that is sufficient to prove local and global regularity of solutions.
There is a number of specific mechanisms that can confer global regularity to an otherwise potentially singular equation.
Perhaps the most common and simple ones are viscosity or diffusion (like for the viscous Burgers equation),
or just strong enough damping. A more sophisticated regularity mechanism which in many instances is nevertheless well understood is dispersion
(see e.g. \cite{LP}). Our focus in this paper is on regularization by fluid flow advection, a purely transport term of the form $(u \cdot \nabla)\rho.$
Fluid flow is naturally present in many settings. It is not expected to be smoothing by itself, but can aid global regularity via mixing/diffusion enhancing,
or dimensionality reduction effects. These phenomena are relatively well understood in the passive advection setting, where the vector field $u$
is given and does not depend on the advected quantity. Examples include certain limiting regimes in homogenization, where the fluid flow leads to
higher renormalized diffusion (e.g. \cite{FP}), flows with good mixing properties that can enhance dissipation (e.g. \cite{CZDE}) or strong shear flows
that can elevate the relative power of diffusion by reducing the effective dimension of the problem (e.g. \cite{BH}). In all these examples, it is important that the
flow has large amplitude. On the other hand, the potential regularizing role of advection is not yet well understood in an active setting, where flow is not
prescribed but determined by an equation that also involves the advected quantity. One fundamental example here is the three dimensional Euler and Navier-Stokes equations,
where global regularity is not known, but there is evidence that model equations for vorticity where the advective part is omitted lead to singularities
(e.g. \cite{PC,hou2012singularity}). Another natural setting is aggregation equations with fluid transport. It is in this area that we discover an intriguing
phenomenon: a variant of the Patlak-Keller-Segel equation, the much studied model coming from mathematical biology and well-known to feature finite time singularity formation, is completely
regularized by an arbitrarily weak coupling to a simple fluid equation.      % }

The Patlak-Keller-Segel equation is a fundamental model of chemotaxis \cite{KS,Patlak}.
It describes a population of bacteria or slime mold
that move in response to attractive external chemical that they themselves secrete.
Here we are interested in its parabolic-elliptic form
\begin{equation}\label{ks1}
\p_t\rho - \Delta \rho + \divv(\rho\nabla c) =0, \,\,\, -\Delta c =\rho - \rho_M, \,\,\,\rho(x,0)=\rho_0(x).
\end{equation}
We can think of \eqref{ks1} set on a finite domain, with the most natural Neumann boundary conditions for $\rho$ and $c;$
in this case $\rho_M$ is the mean of the population density $\rho$, and $c$ is the attractive chemical produced by the bacteria themselves.
Diffusion and production of the chemical $c$ are assumed to be much faster than other time scales of the problem, leading
to the elliptic equation for $c$ in \eqref{ks1}.
The equation \eqref{ks1} has dramatic analytic properties:
in particular, its solutions can form singularities in finite time in dimensions greater than one.
Finite time singularity formation in solutions of the Patlak-Keller-Segel model has been a focus of extensive research (see e.g. \cite{Pert}).
It is known that in two dimensions, finite time blowup is controlled by the mass of $\rho$, with the critical mass being $8\pi$:
all initial data with mass greater than critical lead to a finite time singularity \cite{HV1,HV2, RS, CGMN}, while initial data with mass smaller or equal to
critical lead to globally regular solutions \cite{JL, BDP, BCM}.
%}

Often, chemotactic processes take place in ambient fluid. One natural question is then how the presence of fluid flow can affect singularity formation.
In \cite{KX}, this question was studied for the case where the flow is passive - it is given and does not depend on the density, resulting
in the term $A (u \cdot \nabla) \rho$ added on the left side of \eqref{ks1}.
The main point was a connection between mixing properties of the flow and its ability to suppress chemotactic explosion.
It was proved that given an initial data $\rho_0$ and a flow with strong mixing properties, there exists a flow amplitude $A(\rho_0)$
such that solutions of \eqref{ks1} stay globally regular. The classes of flows for which the result of \cite{KX} holds include relaxation enhancing flows
of \cite{CKRZ}, mixing flows constructed in \cite{YZ}, or self-similar flows of \cite{ACM}.
More research followed: for example, in \cite{BH} chemotactic singularity suppression by strong shear flow has been investigated. In this case the shear works
as a dimension reduction mechanism - and since in one dimension the solutions of the Patlak-Keller-Segel equation are regular, one can expect
suppression of blowup in two dimensions. In \cite{HT}, hyperbolic splitting flow has been studied; in \cite{CDFM} advection induced regularity has been
explored for Kuramoto-Sivashinsky equation. The paper \cite{GXZ} put forward a more general formalism for suppression of singularity formation by fluid advection for a class
of nonlinearities, and considered more general types of passive fluid flows, such as cellular.

A very interesting question that has so far remained largely open is whether fluid advection can suppress chemotactic explosion if it is active, that is, if the Patlak-Keller-Segel equation
is coupled with some fluid mechanics equation. There have been many impressive works that analyzed such coupled systems, usually via buoyancy force;
see for example \cite{di2010chemotaxis,duan2010global,lorz2010coupled,
liu2011coupled,lorz2012coupled,Winkler2012, CKL, DX, TW, Winkler2021} where further references can be found. Active coupling makes the system much more challenging to analyze,
but in some cases results involving global existence of classical solutions
 (the precise notion of their regularity is different in different papers) have been proved.
These results, however, apply either in the settings
where the initial data satisfy some smallness assumptions (e.g. \cite{duan2010global, lorz2012coupled, CKL})
or in the systems where both fluid and chemotaxis equations may not form a singularity if not coupled (e.g. \cite{Winkler2012,  TW, Winkler2021}).
Very recently, in \cite{Hact} and \cite{zzz}, the authors analyzed Patlak-Keller-Segel equation coupled to the Navier-Stokes equation near Couette flow.
Based on ideas of blowup suppression in shear flows and stability of the Couette flow,
the authors proved that global regularity can be enforced if the amplitude of the Couette flow is dominantly large and if the initial flow is very close to it.
%The density/fluid coupling in these works is not by buoyancy force but instead involves a model of the swimmer's effect on fluid that leads to special algebraic properties of the system.

We note that active fluid advection has been conjectured to regularize singular nonlinear dynamics in other settings. The most notable example is the case of the $3$D Navier-Stokes and Euler equations.
Constantin \cite{PC} has proved possibility of finite time singularity formation for the $3$D Euler equation in $\R^3$ if the pure advection term in the vorticity formulation is removed from the model.
Hou and Lei have obtained numerical evidence for finite time blowup in a system obtained from the $3$D Navier-Stokes equation by the removal of the pure transport terms \cite{lei2009stabilizing}.
In fact, finite time blowup has been also proved rigorously in some related modified model settings \cite{hou2011singularity,hou2012singularity}.
Of course, the proof of the global regularity for the $3$D  Navier-Stokes remains an outstanding open problem, so whether the $3$D Navier-Stokes equation exhibits ``advection regularization'' is an open question.
Interestingly, a similar effect is present even in a one-dimensional model of the 3D Euler equation due to De Gregorio \cite{DG}. While the global regularity vs finite time blowup question
is open for this model, the Constantin-Lax-Majda equation that is obtained from it by dropping the advection term leads to finite time singularity formation \cite{CLM}.
Moreover, recent works \cite{JSS} and \cite{Chen} establish global regularity for the full De Gregorio model near the main steady state and under certain symmetry assumptions, respectively.

In this paper, we consider the following system set in $\Omega := \T \times [0,\pi]$:
\begin{equation}
    \label{eq:KSIPM}
    \begin{cases}
        \p_t\rho + u\cdot \nabla\rho - \Delta \rho + \divv(\rho\nabla (-\Delta_N)^{-1}(\rho - \rho_M)) = 0,\\
        u + \nabla p = -g\rho e_2,\; \divv u = 0,\\
        \rho(x,0) = \rho_0(x) \ge 0,\;  \\
                \frac{\p \rho}{\p n}|_{\p \Omega} = 0,\;u\cdot n|_{\p \Omega} = 0,\; x = (x_1,x_2) \in \Omega,
    \end{cases}
\end{equation}
where $\T =(-\pi,\pi]$ is the circle, $\rho_M = \frac{1}{2\pi^2}\int_\Omega \rho dx$, $\Delta_N$ is the Neumann Laplacian and $g\in\mathbb{R}$ is the Rayleigh number representing the strength of buoyancy.
By maximum principle, the solution $\rho(t,x)$ remains nonnegative given that it remains regular.
We keep the same equation for the chemical $c$ as in \eqref{ks1} assuming that its diffusion and production remain faster than other relevant time scales.

The system \eqref{eq:KSIPM} comprises perhaps the most widely studied form of the Patlak-Keller-Segel equation in the parabolic-elliptic version coupled
with a fluid flow evolving according to Darcy's law, a common model of flow in porous media.
 The two equations are coupled via buoyancy
force. Note that when $\rho$ is only advected by $u$ (without diffusion or aggregation), \eqref{eq:KSIPM} becomes the incompressible porous media (IPM) equation, where it is an open question whether smooth initial density can lead to a finite time blowup - although global well-posedness is known when the initial data is close to certain stable
steady states \cite{CGO, ElgindiIPM, CCL}.

Based on the results quoted above regarding chemotaxis suppression by passive advection, one might hope that a similar effect
can be present in \eqref{eq:KSIPM}: given the initial data, one can find sufficiently large $g$ so that chemotactic blowup will be prevented.
The intuition behind such result would be that the buoyancy force tends to stratify the density, inducing and maintaining the configuration close to
one-dimensional and thus changing the balance between chemotaxis and diffusion - resulting in the global regularity.

Surprisingly, it turns out that in fact
chemotactic explosion is prevented for arbitrarily weak coupling $g$: namely, for every $g \ne 0$, for all sufficiently regular initial data, the solutions
stay smooth for all times. The main result of this paper is the following theorem:

\begin{thm}\label{mainthm}
Let $g \ne 0.$  Then for every initial data $\rho_0 \in C^\infty(\Omega),$ the solution to the Patlak-Keller-Segel-IPM system \eqref{eq:KSIPM} is globally regular:
$C^\infty$ in both space and time.
\end{thm}

{\begin{rmk}
It is known from \cite[Theorem 8.1]{KX} that there exists initial data $\rho_0$ with sufficiently large mass for which a blowup occurs for the Keller-Segel equation in a two-dimensional torus $\T^2$. While we consider a slightly different domain $\Omega$, the proof of \cite[Theorem 8.1]{KX} can be adapted to our case via a minor modification, since the blowup happens away from the boundary.
\end{rmk}}

\begin{rmk} \rm 1. In the rest of this paper, we will only consider the case $g>0,$ as the argument for negative $g$ is analogous. \\[0.1cm]
2. We do not pursue the sharpest result in terms of the class of initial data; standard methods allow to consider $\rho_0 \in H^2.$  \\[0.1cm]
3. Our argument can be generalized in several ways to different systems and settings, to be explored in future works. Here we point out
that, for example, changing the $-\Delta c = \rho - \rho_M$ law for the production of chemical to another commonly used law $-\Delta c +c = \rho$
leads to the same result as Theorem \ref{mainthm}. There are some straightforward adjustments to the proof one needs to make, mainly in the energy
estimates of Section~\ref{enest}.\\
{4. It is expected that our main result can be extended to arbitrary compact, smooth domains. While some technical components of this work (e.g. Proposition \ref{prop:concentration326}) relies on Fourier analysis, we expect that these results could still be proven by purely physical-side techniques.}
\end{rmk}

On the intuitive level, one may argue that this phenomenon is possible since even for small $g,$ the density
tending to infinity will make the buoyancy force strong near top concentrations, enabling local stratification that might arrest the blowup.
Such intuition seems hard to turn into a rigorous
argument, and this is not how the proof works. Instead we show, roughly, that for the growth in the $L^2$ norm of density one has to pay with decay in potential energy (to be defined in \eqref{def_potential}).
Given that there is a finite balance of potential energy that one can draw on, it turns out insufficient to grow the $L^2$ norm to infinity.
On the other hand, the $L^2$ norm controls higher regularity, thus preventing any other type of singularity formation as well.

Implementing this plan
involves some nuances and requires understanding of the interaction of several competing mechanisms - diffusion, advection, and chemotaxis.
A key role is played by the observation that the ``main term'' in the derivative of the potential energy in the regime of large $L^2$ norm
is the $H_0^{-1}$ norm squared of $\p_{x_1}\rho.$
Intuitively, due to the finite balance of the potential energy, this means that $\|\p_{x_1} \rho\|^2_{H_0^{-1}}$ should decay. Since decay of the $H^{-1}$
norm of the solution is a well known measure of mixing (see e.g \cite{IKX,LLMND,MMP}), this can be interpreted as solution becoming well-homogenized
via active nonlinear mixing in the $x_1$ direction. Hence, in a certain sense, the solution becomes quasi-one-dimensional and singularity formation
does not happen. We note that the role of the potential energy and $\|\p_{x_1} \rho\|^2_{H_0^{-1}}$ in small scale creation for solutions of the
incompressible porous media equation was first explored by the last two authors in \cite{KY}.

To the best of our knowledge, Theorem~\ref{mainthm} is the first result proving blowup suppression by active advection in a fully nonlinear
setting far away from any perturbative regimes. We stress that in the setting of the system \eqref{eq:KSIPM} we do not expect simple
limiting dynamics, such as convergence to some symmetric steady state that is typical for aggregation equations in regular regimes
(see e.g. \cite{BCL, BCC, CF,  CHVY} and the references therein).
Instead, the dynamics likely stays complex and perhaps turbulent for all times but nevertheless global regularity is preserved.
It is not difficult to upgrade the arguments to obtain global bounds on Sobolev norms of $\rho$ and $u$ ($g$-dependent), but beyond this
it seems challenging to gain more precise information about solutions.

\section{Preliminaries}
\subsection{Notation and conventions}
For any function $f \in L^2(\Omega)$, we denote $f_M$ the mean of $f$ on $\Omega:$ $f_M := \frac{1}{2\pi^2} \int_\Omega f(x,t)\,dx$. We also consider the following decomposition $f = \barf + \tf$, where
 $$
        \barf(x_2) := \frac{1}{2\pi}\int_{\T}f(y_1,x_2) dy_1,\; \tf := f -\barf.
$$
Note that $\barf$ is exactly the projection of $f$ onto the zeroth mode corresponding to direction $x_1$, and $\tf$ lies in the orthogonal complement:
$$
\int_\Omega (\barf(x_2) - f_M)\tf(x_1, x_2)dx_1dx_2 = 0,
$$
{which is due to a stronger fact that $\int_\T \tilde{f}(x_1,x_2)dx_1 = 0$.} Any function $g \in L^p([0,\pi])$, $p \in [1,\infty)$ one can identify with a function in $L^p(\Omega)$ by the correspondence $g(x_1,x_2) = g(x_2)$. One then can connect the one-dimensional and two-dimensional $L^p$ norms by a simple relation $\|g\|_{L^p(\Omega)}^p = 2\pi\|g\|_{L^p([0,\pi])}^p$. For clarity, we will always refer to the two-dimensional $L^2$ norm when we write $\|g\|_{L^2}$. We will use the notation $\|\cdot\|_{L^p([0,\pi])}$ when we emphasize its one-dimensional nature.

The Darcy's law in \eqref{eq:KSIPM} can be simplified by writing $u = \nabla^\perp \psi$ (where $\nabla^\perp =(-\p_{x_2}, \p_{x_1})$), and applying $\nabla^\perp$ to this equation. Note that no flux boundary condition for $u$ corresponds
to the Dirichlet boundary condition for the stream function $\psi.$ Solving for $u$ directly in terms of density results in the relationship
\begin{equation}\label{uBS}
u = g\nabla^\perp(-\Delta_D)^{-1}\p_{x_1}\rho,
\end{equation}
where $\Delta_D$ is the Dirichlet Laplacian.

We will denote a universal constant in the upper bounds by $C$ or $C_i$, a universal constant in the lower bounds by $c$ or $c_i$, and a constant depending on a quantity $X$ by $C(X)$ or $C_i(X)$.
These constants are all positive and subject to change from line to line. For two quantities $A$, $B$, we write $A \lesssim B$ (respectively $A \gtrsim B$) to mean that $A \le CB$  (respectively $A \geq cB$) for some positive constants $C$ and $c$ that may only depends on domain $\Omega$.

\subsection{Functional spaces and eigenfunction expansions}

We denote $\dot H^1(\Omega)$ the homogeneous Sobolev space with semi-norm $\|f\|_1^2 = \int_\Omega |\nabla f|^2\,dx.$
We define the space $\dot H^1_0(\Omega)$ to be the completion of $C^\infty_c(\Omega)$ with respect to norm $\|\cdot\|_1$, and consider its dual space $ \dot H^{-1}_0(\Omega)$. We remark that one can equivalently set
$$
\|f\|_{\dot H^{-1}_0} = \left(\int_\Omega f(-\Delta_D)^{-1}f dx\right)^{1/2}.
$$

Since the problem contains both Dirichlet and Neumann Laplacians, we will find it necessary to use two different eigenfunction expansions.

For $f \in L^2(\Omega),$ define the Dirichlet eigenfunction transform by
\begin{align*}
f(x_1,x_2) &= \frac{1}{\pi} \sum_{k_2=1}^\infty \sum_{k_1 \in \Z}\hat{f}_D(k_1,k_2)e^{ik_1x_1}\sin(k_2 x_2); \\
\hat{f}_D(k_1,k_2) &:= \frac{1}{\pi}\int_\T \int_0^{\pi} f(x_1,x_2)e^{-ik_1 x_1}\sin(k_2 x_2)dx_2dx_1.
\end{align*}
Observe that on the Fourier side,
\[ \|f\|_{\dot H^{-1}_0}^2 = \sum_{k_2=1}^\infty \sum_{k_1 \in \Z}(k_1^2 + k_2^2)^{-1}\big|\hat{f}_D(k_1,k_2)\big|^2. \]

The Neumann eigenfunction transform is given by
\begin{align}
 f(x_1,x_2) &=
\nonumber  \frac{1}{\pi} \sum_{k_2=0}^\infty \sum_{k_1 \in \Z} \frac{1}{1+\delta(k_2)}\hat{f}_N(k_1,k_2)e^{ik_1x_1}\cos(k_2 x_2); \\
\hat{f}_N(k_1,k_2) &:=
\frac{1}{\pi}\int_\T \int_0^{\pi} f(x_1,x_2)e^{-ik_1 x_1}\cos(k_2 x_2)dx_2dx_1, \label{aux1923}
\end{align}
where $\delta(k_2)=1$ if $k_2=0$ and $\delta(k_2)=0$ otherwise.
For a function $f \in \dot H^1(\Omega),$ we have
\[ \|f\|_{\dot H^1}^2 = \sum_{k_2=0}^\infty \sum_{k_1 \in \Z}(k_1^2 + k_2^2)\big|\hat{f}_N(k_1,k_2)\big|^2. \]

%%%%%%%%%%%%%%%%%%%%%%%%%%%%%%%%%%%%%%%%

\section{Local well-posedness and conditional regularity}

Here we comment briefly on the local well-posedness of system \eqref{eq:KSIPM} and state a conditional regularity result which shows that the $L^2$ norm of the density controls possible blowup.
In fact, we remark that there already exists extensive literature which addresses the issue of local well-posedness of Patlak-Keller-Segel equation coupled to various types of fluid equations (e.g. \cite{Winkler2012,Winkler2021}).
For this reason we provide just a rough outline of the proof; more details can be found for example in \cite{Winkler2012}.

\begin{thm}
\label{thm:lwp}
Let $\rho_0 \in C(\Omega)$ be nonnegative. Then there exists a time $T = T(\rho_0) > 0$ such that $\rho(x,t)$ is the unique, nonnegative, regular solution solving \eqref{eq:KSIPM}, with $\rho \in C^\infty(\Omega \times (0,T))$. If $[0,T)$ is the maximal lifespan of $\rho(x,t)$, then
\begin{equation}
\label{Linftycriterion}
\lim_{t \nearrow T}\|\rho(\cdot,t)-\rho_M\|_{L^\infty} = \infty.
\end{equation}
\end{thm}
\begin{proof}
We follow the plan explained in \cite[Lemma 2.1]{Winkler2012}.
Namely, we apply the standard fixed point method, and note that the regularity criterion \eqref{Linftycriterion} arises naturally from the Banach space that we choose.
Consider the Banach space
$$
X := L^\infty([0,T]; C(\Omega)),
$$
{equipped with norm $\|\cdot\|_X = \esssup_{t \in [0,T]}\|\cdot\|_{L^\infty(\Omega)}$.} Here, $T$ will be taken small later according to the initial data. Since each component of $u$ is a Riesz transform of $\rho$, we have the classical Calder\'on-Zygmund estimate:
$$
\|u\|_{L^q} \leq C(q) \|\rho\|_{L^q} \leq C(q)( \|\rho - \rho_M\|_{L^q} + \rho_M)\quad  \text{for any }q\in(1,\infty).
$$
Following the notation in \cite[Lemma 2.1]{Winkler2012}, we consider a fixed-point scheme and define the functional
$$
\Phi(\rho)(x,t) := e^{t\Delta_N}\rho_0 - \int_0^te^{(t-s)\Delta_N}\divv(\rho\nabla(-\Delta_N)^{-1}(\rho - \rho_M) + \rho u)(s,x)ds.
$$
Now, let us fix $q = 6$, $\beta = 1/3$ (note in particular that such choice satisfies the constraint $q \in (2,\infty)$, $\beta \in (1/q, 1/2)$). We follow \cite[equation (2.4)]{Winkler2012} to obtain that for $t \in (0,T)$,
\begin{equation}
\label{est:duhamel}
    \|\Phi(\rho)(\cdot,t)\|_{L^\infty} \le \|\rho_0\|_{L^\infty} + C\int_0^t (t-s)^{-\frac56}\|(\rho\nabla(-\Delta_N)^{-1}(\rho - \rho_M) + \rho u)(\cdot, s)\|_{L^6}ds,
\end{equation}
for some constant $C$. In fact, $C$ originates from standard Neumann heat semigroup estimate, and thus only depends on domain $\Omega$.

Then it suffices for us to control $\|(\rho\nabla(-\Delta_N)^{-1}(\rho - \rho_M) + \rho u)(\cdot, s)\|_{L^6}$ pointwise in time by $\|\rho(\cdot,s)\|_{L^\infty}$. Omitting dependence on time, we observe the following bounds:
\begin{align*}
    \|\rho\nabla(-\Delta_N)^{-1}(\rho - \rho_M)\|_{L^6} &\le \|\rho\|_{L^\infty}\|\nabla(-\Delta_N)^{-1}(\rho - \rho_M)\|_{L^6}\\
    &\lesssim \|\rho\|_{L^\infty}\|\nabla(-\Delta_N)^{-1}(\rho - \rho_M)\|_{\dot H^1}\\
    &\lesssim \|\rho\|_{L^\infty}\|\rho - \rho_M\|_{L^2} \lesssim \|\rho\|_{L^\infty}\|\rho - \rho_M\|_{L^\infty},\\
    \|\rho u\|_{L^6} &\le \|\rho\|_{L^\infty}\|u\|_{L^6} \lesssim \|\rho\|_{L^\infty}\|\rho\|_{L^6} \lesssim \|\rho\|_{L^\infty}^2,
\end{align*}
where we used the Sobolev embedding $\dot H^1 \subset L^6$ in the second inequality and elliptic estimate in the third inequality. Note that all constants involved in the estimates above only depend on domain $\Omega$, since this is the case for embedding inequalities and elliptic estimates. Inserting the estimates above back into \eqref{est:duhamel}, we conclude that
\begin{align*}
\|\Phi(\rho)(\cdot,t)\|_{L^\infty} &\le\|\rho_0\|_{L^\infty} + C(\|\rho\|_X\|\rho - \rho_M\|_X + \|\rho\|_X^2)T^{1/6}\\
&\le \|\rho_0\|_{L^\infty} + C(\|\rho - \rho_M\|_X^2 + \|\rho - \rho_M\|_X)T^{1/6},
\end{align*}
where $C$ only depends on domain $\Omega$. With this estimate, we may follow the argument in \cite[Lemma 2.1]{Winkler2012} to conclude the existence of regular solution and criterion \eqref{Linftycriterion}.
\end{proof}
Notice that the criterion \eqref{Linftycriterion} involves a (spatial) supercritical norm $\|\cdot\|_{L^\infty}$. In the following lemma, we improve this criterion and obtain a refined alternative that only involves the critical norm $\|\cdot\|_{L^2}$. A crucial fact rendering such improvement possible is that an $L^\infty_t L^2_x$ control of $\rho - \rho_M$ can be upgraded to an $L^\infty_t L^\infty_x$ control. Namely, we have
\begin{lem}
\label{lem:L2toLinfty}
Let $\rho_0 \in C^\infty(\Omega)$. Suppose that $\|\rho(\cdot, t) - \rho_M\|_{L^2} \le 2M$ for all $t \in [0,T]$ and some $M \ge \max(1, \sqrt{\rho_M})$. Then we also have $\|\rho(\cdot,t) - \rho_M\|_{L^\infty} \le CM^2$, where $C$ is a universal constant.
\end{lem}
\begin{proof}
    The proof follows verbatim from that of \cite[Proposition 9.1]{KX}, where it is handled for $\T^2$. This is indeed the case since all integrations by parts do not produce any boundary terms due to the Neumann boundary condition that $\rho$ satisfies.
\end{proof}
A refined regularity criterion is then an easy corollary of Theorem \ref{thm:lwp} and Lemma \ref{lem:L2toLinfty}.
\begin{cor}
    \label{cor:L2criterion}
    Suppose that $\rho_0 \in C^\infty(\Omega)$ is nonnegative. Assume $[0,T)$, $T < \infty$, is the maximal lifespan of the unique regular solution $\rho(x,t)$ to \eqref{eq:KSIPM}. Then
    \begin{equation}
        \label{eq:blowupcrit}
        \lim_{t \nearrow T}\|\rho(\cdot,t)-\rho_M\|_{L^2} = \infty.
    \end{equation}
\end{cor}
\begin{proof}
Combining Theorem \ref{thm:lwp} and Lemma \ref{lem:L2toLinfty} immediately imply that
\[
\limsup_{t \nearrow T} \|\rho(\cdot,t)-\rho_M\|_{L^2} = \infty.
\] Next we show that the limsup in the above expression can be replaced by lim.
Denote \[ Y(t) := \|\rho(\cdot,t)-\rho_M\|_{L^2}^2. \]
A simple energy estimate (see Proposition \ref{lem:naive} below) shows that $Y$ satisfies the following differential inequality for a smooth solution:
\[
Y' \leq C(Y^2 +Y) \quad \text{ for } t\in[0,T),
\]
where the constant $C$ may depend on the initial data. For any $0<t<t_k<T$, integrating this differential inequality in $[t,t_k]$ yields
\begin{equation}\label{aux215} Y(t) \geq \frac{1}{(Y(t_k)^{-1}+1)e^{C(t_k-t)}-1 } \quad\text{ for all } t\in [0,t_k]. \end{equation}
If $\limsup_{t\nearrow T}Y(t)=\infty$, by choosing a sequence of time $t_k \nearrow T$ such that $\lim_{k\to\infty} Y(t_k)=\infty$, the above bound becomes
 \[ Y(t) \geq \frac{1}{e^{C(T-t)}-1} \quad\text{ for all } t\in [0,T),\] and $\lim_{t\nearrow T}Y(t)=\infty$ follows.
\end{proof}

\section{Energy estimates}\label{enest}
In this section, we collect a few useful \textit{a priori} estimates.
In what follows, $N_0 \in \N$ will denote a sufficiently large natural number that depends only on the initial data and parameter $g$, and its value may change from line to line. In the course of the proof, we will impose a finite number of conditions on how large $N_0$ needs to be.
\begin{prop}
    \label{lem:mass}
    Assume $\rho(t,x)$ to be a regular, nonnegative solution to \eqref{eq:KSIPM} on $[0,T]$. Then $\|\rho(\cdot,t)\|_{L^1} = \|\rho_0\|_{L^1}=2\pi^2 \rho_M$ for any $t \in [0,T]$.
    \begin{proof}
        The proof easily follows from computing $\frac{d}{dt}\|\rho(\cdot,t)\|_{L^1} = \frac{d}{dt}\int_\Omega \rho dx$ (recall that $\rho \ge 0$) and using incompressibility of $u$.
    \end{proof}
\end{prop}
In the rest of the paper, we will use Proposition \ref{lem:mass} without explicitly referencing it. Next, we introduce an elementary $L^2$ estimate for the density.

\begin{prop}
\label{lem:naive}
    Assume $\rho(t,x)$ is a regular, nonnegative solution to \eqref{eq:KSIPM} on $[0,T]$.
    Then for any $t \in [0,T]$, there exists a universal constant $C$ such that
    \begin{align}
    \label{est:naive1}
        \frac{d}{dt}\|\rho - \rho_M\|_{L^2}^2 + \|\nabla \rho\|_{L^2}^2 \le C\|\rho - \rho_M\|_{L^2}^4 + 2\rho_M \|\rho - \rho_M\|^2_{L^2}.
    \end{align}
\end{prop}
    \begin{proof}
        We remark that energy estimates of this sort are well known in the setting of $\Omega = \T^2$ or $\R^2$ (\cite{KX,Pert}); the argument is very similar in our case
        but we provide a proof for the sake of completeness. Testing the $\rho$ equation of \eqref{eq:KSIPM} on both sides by $\rho - \rho_M$, one obtains the following by applying the divergence theorem:
\begin{align}
\label{eq:L2energy}
    \frac{1}{2}\frac{d}{dt}\|\rho-\rho_M\|_{L^2}^2 + \|\nabla\rho\|_{L^2}^2 &= -\int_\Omega (\rho- \rho_M)\divv(\rho \nabla(-\Delta_N)^{-1}(\rho-\rho_M))dx \notag\\
    &= \frac{1}{2}\int_\Omega \nabla( \rho^2) \cdot\nabla (-\Delta_N)^{-1}(\rho - \rho_M)dx \notag\\
    &= \frac{1}{2}\int_\Omega \rho^2(\rho - \rho_M) dx \notag\\
    &= \frac{1}{2}\int_\Omega (\rho - \rho_M)^3 dx + \rho_M\int_\Omega (\rho - \rho_M)^2 dx
\end{align}
By a Gagliardo-Nirenberg-Sobolev inequality in two dimensions, for any $f\in \dot H^1(\Omega)$ with mean zero we have
$$
\|f\|_{L^3}\le C\|f\|_{L^2}^{2/3} \|\nabla f\|_{L^2}^{1/3}.
$$
Applying this to $f := \rho-\rho_M$,  \eqref{eq:L2energy} becomes
\begin{align*}
    \frac{1}{2}\frac{d}{dt}\|\rho-\rho_M\|_{L^2}^2 + \|\nabla\rho\|_{L^2}^2 &\le C\|\rho - \rho_M\|_{L^2}^2\|\nabla\rho\|_{L^2} + \rho_M \|\rho - \rho_M\|_{L^2}^2\\
    &\le \frac{1}{2}\|\nabla\rho\|_{L^2}^2 + C\|\rho - \rho_M\|_{L^2}^4 + \rho_M \|\rho - \rho_M\|_{L^2}^2.
\end{align*}
Absorbing the gradient term on right side of the above by left side we obtain
$$
\frac{d}{dt}\|\rho-\rho_M\|_{L^2}^2 + \|\nabla\rho\|_{L^2}^2 \le C\|\rho - \rho_M\|_{L^2}^4 + 2\rho_M \|\rho - \rho_M\|_{L^2}^2,
$$
and we conclude the proof.
 \end{proof}

The following corollary follows immediately:
\begin{cor}\label{timecor}
If $\|\rho(\cdot, t) - \rho_M\|_{L^2}^2 \geq 2^{N_0}$ with sufficiently large $N_0$ that  only depends on $\rho_0$ (specifically, on $\rho_M$), then
{
\begin{equation} \label{est:naive}
 \frac{d}{dt}\|\rho(\cdot, t) - \rho_M\|_{L^2}^2 + \|\nabla \rho(\cdot, t)\|_{L^2}^2 \le C\|\rho(\cdot, t) - \rho_M\|_{L^2}^4.
\end{equation}
}
Moreover, if for some $N \geq N_0$ we have $\|\rho(\cdot, s) - \rho_M\|_{L^2}^2 = 2^N$ and $\|\rho(\cdot, r) - \rho_M\|_{L^2}^2 = 2^{N+1}$
for some $r>s,$ then $r-s \geq c_0 2^{-N},$ where $c_0$ is a universal constant.
\end{cor}
\begin{proof}
The estimate \eqref{est:naive} is immediate from the assumption and \eqref{est:naive1}.
To obtain the second statement, let
\[ \tau_2 = {\rm inf} \{t>s: \,\,\|\rho(\cdot, t) - \rho_M\|_{L^2}^2 = 2^{N+1}\}, \]
\[ \tau_1 = {\rm sup} \{t<\tau_2: \,\,\|\rho(\cdot, t) - \rho_M\|_{L^2}^2 = 2^{N}\}. \]
For $t \in [\tau_1,\tau_2]$ we clearly have $2^N \leq \|\rho(\cdot, t) - \rho_M\|_{L^2}^2 \leq 2^{N+1}$ and so \eqref{est:naive} holds given $N > N_0$, where $N_0$ is sufficiently large. Then $r-s \geq \tau_2-\tau_1 \geq c_0 2^{-N},$ {which follows from solving the differential inequality $X' \lesssim X^2$.}
\end{proof}

Next, we give a variant of \eqref{est:naive}, which takes advantage of the orthogonal decomposition $\rho = \brho + \trho$.
\begin{prop}\label{propen}
    Assume $\rho(x,t)$ is a regular, nonnegative solution to \eqref{eq:KSIPM} on $[0,T]$.
    Then for any $t \in [0,T]$, there exist universal constants $C_2$ and $C_3$ such that
    \begin{align}
    \label{est:mainenergy1}
    \frac{d}{dt}\|\rho - \rho_M\|_{L^2}^2 &\le -\frac{C_2}{\rho_M^4}\|\brho - \rho_M\|_{L^2}^6+ C_3 \left( \|\brho - \rho_M\|_{L^2}^{10/3} + \rho_M \|\brho - \rho_M\|_{L^2}^{2}+ \rho_M\|\trho\|_{L^2}^2 + \|\trho\|_{L^2}^4\right).
    \end{align}
\end{prop}
    \begin{proof}
We test \eqref{eq:KSIPM} on both sides by $\rho - \rho_M$, which yields
\begin{align*}
    \frac{1}{2}\frac{d}{dt}\|\rho-\rho_M\|_{L^2}^2 +  (\|\p_{x_2} \brho\|_{L^2}^2 + \|\nabla \trho\|_{L^2}^2) &= \frac{1}{2}\int (\brho - \rho_M + \trho)^3 + \rho_M\int (\brho - \rho_M + \trho)^2: = I+J.
\end{align*}
First, one observes that $J = \rho_M(\|\brho - \rho_M\|_{L^2}^2 + \|\trho\|_{L^2}^2)$ using orthogonality.
To estimate $I$, using the elementary inequality $(a+b)^3 \leq C(|a|^3 + |b|^3)$, we have
\[
I \leq C\left( \int_{\Omega}  |\brho - \rho_M|^3  dx +\int_{\Omega}  |\trho|^3 dx\right) =: C(I_1 + I_2).
\]

Before proceeding, recall the following Gagliardo-Nirenberg-Sobolev inequality for $d = 1,2$ for mean-zero $f$:
$$
\|f\|_{L^3}\le C(d)\|f\|_{L^2}^{1-\frac{d}{6}} \|\nabla f\|_{L^2}^{\frac{d}{6}},
$$
where $C(d)$ is some dimensional constant. Applying the above inequality to $I_1$ (with $f=\brho-\rho_M$ and $d=1$) and $I_2$ (with $f=\tilde\rho$ and $d=2$) respectively, we have the following for any $\epsilon>0$:
\begin{align*}
     {I_1} & \lesssim \|\brho - \rho_M\|_{L^2([0,\pi])}^{5/2}\|\p_{x_2}\brho\|_{L^2([0,\pi])}^{1/2} \le \epsilon \|\p_{x_2}\brho\|_{L^2}^2 + C(\epsilon)\|\brho - \rho_M\|_{L^2}^{10/3};\\[0.1cm]
    I_2 & \lesssim \|\trho\|_{L^2}^2 \|\nabla \trho\|_{L^2} \le \epsilon \|\nabla \trho\|_{L^2}^2 + C(\epsilon)\|\trho\|_{L^2}^4.
\end{align*}
Choosing $\epsilon$ small enough and combining the estimates of $I, J$, we obtain
\begin{align}\label{aux13a}
    \frac{d}{dt}\|\rho - \rho_M\|_{L^2}^2 + \|\p_{x_2} \brho\|_{L^2}^2 + \|\nabla \trho\|_{L^2}^2 & \le C_3 \left( \|\brho - \rho_M\|_{L^2}^{10/3} + \rho_M\|\brho - \rho_M\|_{L^2}^2+
    \rho_M\|\trho\|_{L^2}^2 + \|\trho\|_{L^2}^4 \right).
\end{align}
By the one-dimensional Nash's inequality, we may bound the diffusion term from below by
\begin{align*}
    \|\p_{x_2}\brho\|_{L^2}^2 = 2\pi \|\p_{x_2}\brho\|_{L^2([0,\pi])}^2 \gtrsim \|\brho - \rho_M\|_{L^1}^{-4}\|\brho - \rho_M\|_{L^2}^6.
\end{align*}
Moreover, we have the elementary bound $\|\brho - \rho_M\|_{L^1} \lesssim \rho_M$. Combining with the energy inequality \eqref{aux13a}, we finally obtain \eqref{est:mainenergy1}.
\end{proof}

The following corollary simplifies the form of the differential inequality in large $L^2$ norm regime.
\begin{cor}\label{coren}
In addition to assumptions of Proposition \ref{propen}, suppose that $\|\rho(\cdot,t) - \rho_M\|_{L^2}^2 \ge 2^{N_0}$ for $t \in [s,r]$ and sufficiently large $N_0$ that may only depend on $\rho_M$.
Then for any $t \in [s,r]$ we have
 \begin{align}
    \label{est:mainenergy}
    \frac{d}{dt}\|\rho - \rho_M\|_{L^2}^2 &\le -\frac{C_2}{\rho_M^4}\|\brho - \rho_M\|_{L^2}^6+ C_3 \|\trho\|_{L^2}^4
    \end{align}
    for some universal constants $C_2$ and $C_3$.
\end{cor}
\begin{proof}
Recall that due to orthogonality,
\[ \|\rho - \rho_M\|_{L^2}^2 = \|\bar \rho -\rho_M\|_{L^2}^2 + \|\tilde{\rho}\|_{L^2}^2. \]
If $\|\bar \rho - \rho_M\|_{L^2}^2 \geq \frac12 \|\rho - \rho_M\|_{L^2}^2,$ then, provided $N_0$ is sufficiently large, the first term
on the right side of \eqref{est:mainenergy1} dominates all other terms. If $\|\bar \rho - \rho_M\|_{L^2}^2 \leq \frac12 \|\rho - \rho_M\|_{L^2}^2,$
then $\|\tilde{\rho}\|_{L^2} \geq \frac12 \|\rho - \rho_M\|_{L^2}^2.$ In this case, if $N_0$ is sufficiently large, the last term on the right side
of \eqref{est:mainenergy1} dominates all other terms with possible exception of the first one.
\end{proof}

At the end of this section, we introduce a potential energy of the system \eqref{eq:KSIPM}, which will play a crucial role later. Define
\begin{equation}\label{def_potential}
E(t) := \int_\Omega \rho(x,t) x_2 dx.
\end{equation}
Note that if $\rho(x,t)$ is a nonnegative solution of \eqref{eq:KSIPM} that blows up at $T_* < \infty$, then $E(t) \ge 0$ for $t \in [0,T_*)$,
and
\begin{equation}\label{aux13e}
\sup_{t \in [0,T_*)}E(t) \le \pi\|\rho(\cdot,t)\|_{L^1} = \pi\|\rho_0\|_{L^1}
\end{equation}
due to conservation of mass. In addition we need to control the rate of change of the potential energy, and it is here that the IPM part of \eqref{eq:KSIPM} is genuinely exploited:

\begin{prop}\label{potender}
For any $t \in (0,T_*)$, we have
\begin{align}
\label{eq:dtE}
E'(t) = \underbrace{-g\|\p_{x_1} \rho\|_{\dot H^{-1}_0}^2}_{\text{Main term}} - \underbrace{\int_\Omega \p_{x_2} \rho dx}_{\text{Diffusion}} + \underbrace{\int_\Omega \rho\p_{x_2}(-\Delta_N)^{-1}(\rho - \rho_M)}_{\text{Keller-Segel nonlinearity}}.
\end{align}
\end{prop}
\begin{proof}
The proposition follows from a straightforward computation. Using integration by parts and the representation \eqref{uBS}, we compute
\begin{align*}
%\label{eq:dtE}
    E'(t) &= \int_\Omega \rho u_2 dx - \int_\Omega \p_{x_2} \rho dx + \int_\Omega \rho\p_{x_2}(-\Delta_N)^{-1}(\rho - \rho_M)dx \notag\\
    &= -g\int_\Omega \p_{x_1}\rho (-\Delta_D)^{-1}\p_{x_1}\rho dx - \int_\Omega \p_{x_2} \rho dx + \int_\Omega \rho\p_{x_2}(-\Delta_N)^{-1}(\rho - \rho_M)dx, \notag %\\
  %  &= -g\|\p_{x_1} \rho\|_{\dot H^{-1}_0}^2 +\int_\Omega \p_{x_2} \rho dx -\int_\Omega \rho\p_{x_2}(-\Delta_N)^{-1}(\rho - \rho_M)dx,
\end{align*}
which is exactly \eqref{eq:dtE} by the definition of  $\|\cdot\|_{\dot H^{-1}_0}$ norm.
%where we used the definition of $\|\cdot\|_{\dot H^{-1}_0}$ norm in the last equality.
\end{proof}

%%%%%%%%%%%%%%%%%%%%%%%%%%%%%%%%%%%%%%%%%%%%%%%%%%%%%%%%%%%%%%%%%%%%%%%%
\section{Partition into ``good'' and ``bad'' time intervals}

In this section, we suppose that $\rho(t,x)$ forms a finite time singularity at $T_* < \infty$. Let $N_0 \in \N$ be sufficiently large, in particular so that $2^{N_0-1} > \|\rho_0 - \rho_M\|_{L^2}^2$.
Let us define
\[ t_1 := \sup \{ t \in [0,T_*)\;|\; \|\rho(\cdot,t)-\rho_M\|_{L^2}^2 = 2^{N_0} \}. \]
{Note that $t_1$ is well-defined thanks to the regularity criterion established in Corollary \ref{cor:L2criterion}, which rules out oscillations in time near $T_*$.} Inductively, given $t_k$ such that $\|\rho(\cdot,t_k)-\rho_M\|_{L^2}^2 = 2^{N}$ for some $N \geq N_0,$ let us define
$t_{k+1} \in (t_k, T_*)$ to be the smallest time such that either $\|\rho(\cdot,t_{k+1})-\rho_M\|_{L^2}^2 = 2^{N-1}$ or $\|\rho(\cdot,t_{k+1})-\rho_M\|_{L^2}^2 = 2^{N+1}.$
It is clear that an infinite sequence
$t_k,$ $k=1,2,...$ is well defined %and $\cup_{k=1}^\infty (t_k,t_{k+1}) = [t_1,T_*)$
since by Corollary~\ref{cor:L2criterion} we see that the $L^2$ norm of $\rho$ has to blowup at $T_*$.
It is also clear that for $t \in (t_k, t_{k+1}),$ we have $2^{N-1}<\|\rho(\cdot,t)-\rho_M\|_{L^2}^2 < 2^{N+1}.$  See Figure~\ref{fig_intervals} for an illustration of the time sequence $\{t_k\}_{k\in\mathbb{N}}$.

Once the sequence of times $\{t_k\}_{k\in\mathbb{N}}$ is determined as above, for each time interval $(t_k, t_{k+1})$, we will classify its \emph{level} based on the values of $\|\rho(\cdot,t)-\rho_M\|_{L^2}^2$ at the two endpoints, and name it either as ``good'' or ``bad'' from the perspective of a finite time singularity formation: the ``good intervals'' contribute to the blowup of $L^2$ norm of $\|\rho(\cdot,t)-\rho_M\|_{L^2}^2$; whereas on the ``bad intervals'' the norm gets driven down.
 The precise definition is as follows:
\begin{definition}
For $k\in\mathbb{N}$ and $N\geq N_0$, we call an interval $(t_k,t_{k+1})$ a good interval of level $N$ if $\|\rho(\cdot,t_k)-\rho_M\|_{L^2}^2 = 2^{N}$ and
$\|\rho(\cdot,t_{k+1})-\rho_M\|_{L^2}^2 = 2^{N+1}.$ We call an interval $(t_k,t_{k+1})$ a bad interval of level $N$ if $\|\rho(\cdot,t_k)-\rho_M\|_{L^2}^2 = 2^{N+1}$ and
$\|\rho(\cdot,t_{k+1})-\rho_M\|_{L^2}^2 = 2^{N}.$
\end{definition}

See Figure~\ref{fig_intervals} for an illustration of good and bad intervals at various levels.
Note that on a good interval of level $N$ we have $2^{N-1}<\|\rho(\cdot,t)-\rho_M\|_{L^2}^2 < 2^{N+1}$; whereas on a bad interval of level $N$ we have $2^{N}<\|\rho(\cdot,t)-\rho_M\|_{L^2}^2 < 2^{N+2}$.

\begin{figure}[htb]
\begin{center}
\includegraphics[scale=1]{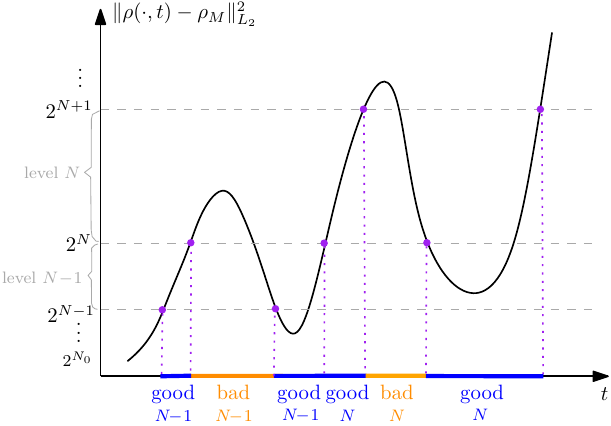}
\caption{
\label{fig_intervals}
The times $\{t_k\}$ are marked by the dotted vertical lines. On the $t$ axis, the good intervals are marked in blue color, and the bad intervals are marked in orange color.}
\end{center}
\end{figure}

\begin{lem}\label{timeint}
Suppose that $\rho(t,x)$ forms a finite time singularity at $T_* < \infty.$
Then for any $N\geq N_0$, there can only be a finite number of good and bad time intervals of a given level $N$ in $[0,T_*).$
These level-$N$ intervals intertwine, and the number of good intervals of level $N$ must exceed the number
of bad time intervals of level $N$ by one.
\end{lem}
\begin{proof}
It is clear that the first interval of level $N$ must be good; it starts at the smallest $t_k \in [t_1,\infty)$ such that $\|\rho(\cdot,t_k)-\rho_M\|_{L^2}^2=2^N$
and after $t_k$ the norm $\|\rho(\cdot,t)-\rho_M\|_{L^2}^2$ reaches $2^{N+1}$ at a time $t_{k+1}$ before it reaches $2^{N-1}.$

To show the good and bad intervals of level $N$ intertwine, we start by showing that between any two good intervals of level $N$, there exists a bad interval of level $N$. To see this, observe that by definition, the good interval of level $N$ is followed either by a good interval of level $N+1$ or a bad interval of level $N,$
while a bad interval of level $N$ is followed either by a good interval of level $N$ or by a bad interval of level $N-1.$ Hence we can arrive at the
next good interval of level $N$ only from a bad interval of level $N$ or a good interval of level $N-1.$ But we can descend the levels only along bad intervals,
and we cannot make any jumps. Thus to arrive at a good interval of level $N-1,$ we still have to go through some bad interval of level $N.$
A similar reasoning applies to bad intervals of level $N$: there always has to be a good interval of level $N$ between them.

Next we show that the total number of level $N$ intervals is finite. Observe that by Corollary~\ref{timecor}, any good interval of level $N$ has length at least $c_0 2^{-N}$.
Hence there can only be a finite number of good intervals of level $N$ in $[t_1,T_*)$. Since we have shown that the good and bad level $N$ intervals intertwine,  the total number of level $N$ intervals must also be finite.

We have shown in the beginning that the first  interval of level $N$ must be good. In fact, the last one must also be good, otherwise $\rho$ would not form a finite time blowup. Combining this with the intertwining of good and bad interval of level $N$, we know the number of good intervals of level $N$ exceeds the number of bad intervals of level $N$ by one.
\end{proof}

If $(t_k, t_{k+1})$ is a level $N$ interval (can be either good or bad), from the definition we immediately have $\int_{t_k}^{t_{k+1}} \|\rho-\rho_M\|_{L^2}^2 dt \sim 2^N (t_{k+1}-t_k)$.
Recall that $\|\rho-\rho_M\|_{L^2}^2 = \|\brho - \rho_M\|_{L^2}^2 + \|\trho\|_{L^2}^2 $, where the two terms on the right hand side play opposite roles in the growth of $L^2$ norm: see \eqref{est:mainenergy}.
For this reason, our consideration will depend on whether $\int_{t_k}^{t_{k+1}} \|\trho\|_{L^2}^2 dt \sim 2^N (t_{k+1}-t_k)$ holds. The next lemma shows that this always holds in a good interval. For a bad interval this might not be true; but in this case $t_{k+1}-t_k$ must be very short.

\begin{lem}\label{lem_trho}
Let $N_0$ be sufficiently large, which only depends on $\rho_M$.
\begin{enumerate}
\item[(a)]
Assume that $N\geq N_0$ and  $(t_k, t_{k+1})$ is a good interval of level $N$. Then $\trho$ satisfies
\begin{equation}\label{eq_lemma1}
\int_{t_k}^{t_{k+1}} \|\trho\|_{L^2}^2 \,dt \geq 2^{N-2} (t_{k+1}-t_k).
\end{equation}
\item[(b)]
Assume that $N\geq N_0$ and  $(t_k, t_{k+1})$ is a bad interval of level $N$. Assume in addition that $\trho$ does not satisfy \eqref{eq_lemma1}. Then $t_{k+1}-t_k < C \rho_M^4 2^{-2N}$ for some universal constant $C>0$.
\end{enumerate}
\end{lem}

\begin{proof}
For any (good or bad) interval $(t_k, t_{k+1})$ of level $N$, assume that  \eqref{eq_lemma1} fails on this interval, i.e.
\begin{equation}\label{eq_lemma}
\int_{t_k}^{t_{k+1}} \|\trho\|_{L^2}^2 \, dt < 2^{N-2} (t_{k+1}-t_k).
\end{equation}
We first aim to show that $(t_k, t_{k+1})$ cannot be a good interval.  Note that $\|\rho(\cdot,t)-\rho_M\|_{L^2}^2 > 2^{N-1}$ for $t\in(t_k, t_{k+1})$.
Combining this with \eqref{eq_lemma} and the fact that $\|\rho-\rho_M\|_{L^2}^2 = \|\brho - \rho_M\|_{L^2}^2 + \|\trho\|_{L^2}^2 $ gives
\[
\int_{t_k}^{t_{k+1}} \|\brho-\rho_M\|_{L^2}^2 \,dt > \int_{t_k}^{t_{k+1}} (2^{N-1}- \|\trho\|_{L^2}^2  )\, dt > 2^{N-2} (t_{k+1}-t_k).
\]
Applying H\"older's inequality to the above yields
\[
\begin{split}
\int_{t_k}^{t_{k+1}} \|\brho-\rho_M\|_{L^2}^6 \, dt &\geq \left(\int_{t_k}^{t_{k+1}} \|\brho-\rho_M\|_{L^2}^2 \, dt\right)^3 (t_{k+1}-t_k)^{-2}\\
& > 2^{3N-6} (t_{k+1}-t_k).
\end{split}
\]
Also note that \eqref{eq_lemma} and the fact that $\|\trho\|_{L^2}^2\leq \|\rho-\rho_M\|_{L^2}^2 \leq 2^{N+2}$ in $(t_k,t_{k+1})$ imply
\[
\int_{t_k}^{t_{k+1}} \|\trho\|_{L^2}^4 dt \leq 2^{2N+4} (t_{k+1}-t_k).
\]
Integrating \eqref{est:mainenergy} in $(t_k, t_{k+1})$ and using the above two inequalities, we have
\begin{equation}\label{temp3}
\|\rho(\cdot,t_{k+1})-\rho_M\|_{L^2}^2 - \|\rho(\cdot, t_{k}) -\rho_M\|_{L^2}^2 < \left(-2^{3N-6} C_2 \rho_M^{-4} + 2^{2N+4} C_3 \right) (t_{k+1}-t_k).
\end{equation}
Let $N_0$ be sufficiently large such that $2^{N_0-11} C_2 \rho_M^{-4} \geq C_3$. For any $N\geq N_0$, we have
\begin{equation}\label{temp4}
%\text{RHS of }\eqref{temp3}
\|\rho(\cdot,t_{k+1})-\rho_M\|_{L^2}^2 - \|\rho(\cdot,t_{k})-\rho_M\|_{L^2}^2  < -2^{3N-7} C_2 \rho_M^{-4} (t_{k+1}-t_k) < 0.
\end{equation}
Thus $(t_k, t_{k+1})$ cannot be a good interval: for a good interval we have that the left hand side of \eqref{temp4} is $2^N$, which contradicts \eqref{temp4}.
Finally, if $(t_k, t_{k+1})$ is a bad interval of level $N$, then \eqref{temp4} imply
\[
-2^N < -2^{3N-7} C_2 \rho_M^{-4} (t_{k+1}-t_k).
\]
Thus $t_{k+1}-t_k < C \rho_M^4 2^{-2N}$, finishing the proof.
\end{proof}
%%%%%%%%%%%%%%%%%%%%%%%%%%%%%%%%%%%%%%%%%%%%%%%%%%%%%%%%%%%%%%%%%%%%%%%%%

\section{Main estimates}
\subsection{Key lemmas}
We first collect a few technical results that will be fundamental in the estimates on both good and bad intervals. %\begin{lem}

The following proposition applies to any $\rho \in \dot H^1(\Omega)$ and can be of independent interest.
While we will prove it for the set $\Omega$, it has straightforward (and in fact simpler to prove) analogs when the domain does not have a boundary, for example in the $\T^2$ or $\R^2$ case.
\begin{prop}
\label{prop:concentration326}
Assume that $\rho\in \dot{H}^1(\Omega)$. For some $N\geq 1$, assume that
    \begin{equation}\label{Hassump}
        \|\p_{x_1} \rho\|_{\dot H^{-1}_{0}(\Omega)}^2  \leq  N^{-1} \|\trho\|_{L^2(\Omega)}^2.
    \end{equation}
Then
        \begin{equation}\label{GNN1}
\|\trho\|_{L^2(\Omega)}^2 \le CN^{-1/4} \|\trho\|_{L^1(\Omega)} \|\nabla \trho\|_{L^2(\Omega)}
        \end{equation}
        for some universal constant $C.$
\end{prop}
\begin{rmk}  Observe that this Proposition can be restated in the following form: 
\[ \|\trho\|_{L^2(\Omega)} \le C \|\p_{x_1} \rho\|_{\dot H^{-1}_{0}(\Omega)}^{1/5} \|\trho\|_{L^1(\Omega)}^{2/5} \|\nabla \trho\|_{L^2(\Omega)}^{2/5}, \]
giving a stronger generalized 2D Nash inequality for $\trho.$ As the norm $\|\p_{x_1} \rho\|_{\dot H^{-1}_{0}(\Omega)}$ appears naturally in the IPM setting, 
we expect the inequality to be useful in further analysis of this equation. 
\end{rmk}

One can think of the proposition as a quantitative improvement of the Nash's inequality.  Without the assumption \eqref{Hassump}, the Nash's inequality exactly looks like \eqref{GNN1} with some order-one coefficient on the right hand side. We aim to show that the coefficient can be made arbitrarily small (namely, into $C N^{-1/4}$) if $\rho$ satisfies \eqref{Hassump},
which can be understood as some sort of ``mixing'' assumption on $\rho$ in one direction. Indeed, the $H^{-1}$ norm is often used as a measure of mixing (see e.g. \cite{IKX,LLMND,MMP}).
In our context this assumption will imply that $\rho$ is more regular in $x_1$ than in $x_2$: specifically, on the Fourier side much of its $L^2$ norm is supported in a narrow cone along the $k_2$-axis.
This will allow an improvement in the constant in the Nash inequlity \eqref{GNN1} that is crucial to establish Theorem~\ref{mainthm}.

\begin{proof}
In proving this proposition, we will have to negotiate bounds in Dirichlet and Neumann Laplacian eigenfunction expansions. Indeed, the $\dot H^{-1}_0$ norm assumption leads naturally to information best stated in the Dirichlet context  - but $\trho$ does not satisfy
the Dirichlet boundary condition and to obtain an accurate estimate involving $\|\nabla \trho\|_{L^2}$ we will have to relay the bounds to the Neumann basis.

Define the cones \[ \mC_1 = \{k = (k_1,k_2) \in \Z^2\;|\; k_2 \ge \aaa N^{1/2}|k_1|\} \] and \[ \mC_2 = \{k = (k_1,k_2) \in \Z^2\;|\; k_2 \ge \aaa^2 N^{1/2}|k_1|\},\]
where $\aaa\in(0,1)$ is a small parameter to be fixed later. The choice of $\aaa$ will be universal, not depending on any parameters of the problem. See Figure~\ref{cones} for an illustration of the cones.

\begin{figure}[htb]
\begin{center}
\includegraphics[scale=1]{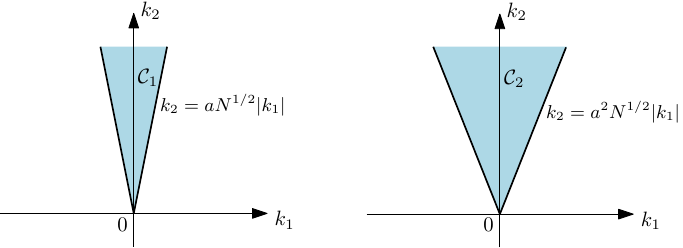}
\caption{
\label{cones}An illustration of the cones in the definitions of $\mathcal{C}_1$ and $\mathcal{C}_2$ for some $N\gg 1$. Note that $\mathcal{C}_1$ and $\mathcal{C}_2$ are subsets of $\mathbb{Z}^2$, thus they only contain the integer points $k\in\mathbb{Z}^2$ that fall in the shaded regions.}
\end{center}
\end{figure}

Let us denote $P_1^D$ and $P_1^N$ the Fourier side orthogonal projectors on $\mC_1$ in the Dirichlet and on $\mC_2$ in the Neumann expansions, respectively:
\[ \widehat{(P_1^D f)}_D (k_1,k_2) = \hat{f}_D(k_1,k_2) \chi_{\mC_1}(k_1,k_2); \,\,\,\widehat{(P_1^N f)}_N (k_1,k_2) = \hat{f}_N(k_1,k_2) \chi_{\mC_2}(k_1,k_2).  \]

Let us split $\trho$ in two different ways into sums of two orthogonal components:
\[ \trho(x,t) = P_1^D \trho(x,t) + (I-P_1^D) \trho(x,t) := \trho_1^D(x,t)+\trho_2^D(x,t) \]
and
\[ \trho(x,t) = P_1^N \trho(x,t) + (I-P_1^N) \trho(x,t) := \trho_1^N(x,t)+\trho_2^N(x,t), \]
where $I$ is the identity map. Observe first that on the complement of the first cone, $\mC_1^c,$ we have
\[ \frac{k_1^2}{k_1^2 + k_2^2} \geq \frac{\aaa^{-2} N^{-1}}{\aaa^{-2} N^{-1}+ 1} \geq \frac12 {\rm min}(1,\aaa^{-2} N^{-1}). \]
At first, we will assume that $N$ is so large that  $\aaa^{-2} N^{-1} <1;$ we will explain in the end how to handle the smaller values of $N.$
Combining the above with the assumption \eqref{Hassump}, we have
\[ \frac{\aaa^{-2} N^{-1}}{2} \|(I-P_1^D)\trho\|_{L^2}^2  \leq \|\p_{x_1} \trho\|_{\dot H^{-1}_0}^2  \leq  N^{-1}\|\trho\|_{L^2}^2. \]
Hence
\begin{equation}\label{aux18d}   \|\trho_2^D\|_{L^2}^2  \leq 2 \aaa^2 \|\trho\|_{L^2}^2, \end{equation}
and therefore
\begin{equation}\label{Dexpinfo} \|\trho_1^D\|_{L^2}^2 \geq (1-2 \aaa^2) \|\trho\|_{L^2}^2; \end{equation}
here we used orthogonality of $\trho_1^D$ and $\trho_2^D.$

The estimate \eqref{Dexpinfo} shows that if $\aaa$ is sufficiently small, $\trho_1$ is mostly supported on $\mC_1$ in the Dirichlet expansion setting.
Next we translate this information into the Neumann setting. Consider
\begin{align}\label{aux18a}
(I-P_1^N) \trho_1^D(x_1,x_2) = \frac{1}{\pi} \sum_{k \in \mC_2^c, \,k_2 \geq 0} \frac{1}{1+\delta(k_2)} e^{ik_1x_1} \cos k_2 x_2 \int_\T \int_0^{\pi} \trho_1^D(y_1,y_2)e^{-ik_1 y_1}\cos(k_2 y_2)dy_2dy_1,
\end{align}
where  $\delta(k_2)=1$ if $k_2=0,$ and $0$ otherwise (due to normalization constant, as in \eqref{aux1923}), and
\begin{equation}\label{aux18b} \trho_1^D(y_1,y_2) = \frac{1}{\pi} \sum_{\tau \in \mC_1} \hat{\trho}_D(\tau_1,\tau_2)e^{i\tau_1 y_1}\sin(\tau_2 y_2). \end{equation}
Substituting \eqref{aux18b} into \eqref{aux18a} and integrating out in $y_1,$ we obtain
\[ (I-P_1^N) \trho_1^D(x_1,x_2) = \frac{2}{\pi} \sum_{k \in \mC_2^c, \,k_2 \geq 0} \frac{1}{1+\delta(k_2)} e^{ik_1x_1} \cos k_2 x_2 \sum_{\tau_2 \geq \aaa N^{1/2} |k_1|} \hat{\trho}_D(k_1,\tau_2)
\int_0^\pi \sin \tau_2 y_2 \cos k_2y_2 \,dy_2. \]
Note that in the second sum above, there are no terms corresponding to $k_1=0$ since $\hat{\trho}_D(0,\tau_2)=0$ for all $\tau_2$
due to definition of $\trho.$
Integration by parts then shows that
\[ \left| \int_0^\pi \sin \tau_2 y_2 \cos k_2y_2 \,dy_2 \right| \leq \frac{2\tau_2}{\tau_2^2-k_2^2} \leq \frac{2}{|\tau_2| - |k_2|} \leq \frac{4}{|\tau_2|} \]
provided that $\aaa \leq \frac12$ (which implies $\tau_2 > 2k_2$).
Then
\begin{equation}
\begin{split}\label{aux18e}
 \left\|(I-P_1^N) \trho_1^D\right\|_{L^2}^2 &\leq C \sum_{k_1 \in \Z} \,\, \sum_{0 \leq k_2 < \aaa^2 N^{1/2} |k_1|} \left| \sum_{\tau_2 \geq \aaa N^{1/2} |k_1|}
\hat{\trho}_D(k_1,\tau_2) \frac{1}{\tau_2} \right|^2   \\
& \leq C \aaa^2 N^{1/2} \sum_{k_1 \in \Z\setminus\{0\}} |k_1|
 \left( \sum_{\tau_2 \geq \aaa N^{1/2} |k_1|} |\hat{\trho}_D(k_1,\tau_2)|^2 \right) \left( \sum_{\tau_2 \geq \aaa N^{1/2} |k_1|}  \frac{1}{\tau_2^2}\right) \\
 &\leq  C \aaa^2 N^{1/2} \sum_{k_1 \in \Z\setminus\{0\}} |k_1| \sum_{\tau_2 \geq \aaa N^{1/2} |k_1|} |\hat{\trho}_D(k_1,\tau_2)|^2 \frac{1}{\aaa N^{1/2} |k_1|} \\
&\leq C \aaa \sum_{k_1 \in \Z} \,\, \sum_{\tau_2 \geq \aaa N^{1/2} |k_1|} |\hat{\trho}_D(k_1,\tau_2)|^2 \,dt \leq C\aaa \|\trho\|_{L^2}^2.
\end{split}
\end{equation}
Here in the first step we used the Parseval identity, in the second step summed over $k_2$ and applied the Cauchy-Schwarz inequality, and then just simplified the resulting bound.
Combining \eqref{aux18e} with \eqref{aux18d} we find that with an adjusted universal constant $C,$
\begin{equation}\label{aux18f}  \|(I-P_1^N) \trho\|_{L^2}^2  \leq C\aaa \|\trho\|_{L^2}^2 \leq \frac12 \|\trho\|_{L^2}^2;  \end{equation}
the last step follows if $\aaa$ is sufficiently small. Thus in the Neumann expansion, we have that $\trho$ is also mostly supported in a slightly larger, but still narrow cone.

Now let $\lambda > 0$ be a parameter. For any $\lambda>0$, we observe that
            \begin{align}
               \nonumber \|\trho\|_{L^2}^2 &=  \sum_{k \in \mC_2,\;0< |k| < \lambda}|\hat\trho_N(k)|^2 + \sum_{k \not\in \mC_2,\;|k| < \lambda}|\hat \trho_N(k)|^2+ \sum_{|k| \ge \lambda} |\hat \trho_N(k)|^2 \\
                \nonumber &\le  \lambda^2 \aaa^{-2} N^{-1/2}\|\hat \trho_N\|_{l^\infty(\Z^2)}^2 + \sum_{k \not\in \mC_2,\;|k| < \lambda}|\hat \trho_N(k)|^2 + \lambda^{-2}\sum_{k}|k|^2|\hat \trho_N(k)|^2 \\
              \label{aux18g}  &\le \lambda^2 \aaa^{-2} N^{-1/2}\|\trho\|_{L^1}^2 + \frac{1}{2}\|\trho\|_{L^2}^2+ \lambda^{-2}\|\nabla \trho\|_{L^2}^2.
            \end{align}
            Note that we used the estimate $|\{k \in \mC_2\} \cap \{|k| < \lambda\}| \le \lambda^2 \aaa^{-2} N^{-1/2}$ in the first inequality, and $\|\hat \trho\|_{l^\infty(\Z^2)} \le \|\trho\|_{L^1}$ with \eqref{aux18f} in the second inequality.
           Also, $\int_\Omega |\nabla f|^2\,dx$ defined on $H^1(\Omega)$ is the quadratic form corresponding to the Neumann Laplacian $-\Delta_N$, and hence $\sum_k |k|^2 |\hat f_N(k)|^2 = \|\nabla f\|^2_{L^2}$ for all $f \in H^1(\Omega).$
            Observe that all bounds we imposed on $\aaa$ were universal, and are independent on $N$ or $\lambda$. Then for sufficiently large $N$ (we only require $N>\aaa^{-2}$), the desired bound \eqref{GNN1} follows from rearranging the final line above and optimizing over $\lambda,$ absorbing $\aaa^{-1}$ into the constant.
            For smaller $N$, the bound \eqref{GNN1} follows from the usual Nash inequality, and these $N$ can be absorbed by increasing the value of constant $C$ in \eqref{GNN1} if necessary.

            Note that if we tried to run the argument leading to \eqref{aux18g} only using the Dirichlet expansion, it would not imply \eqref{GNN1}: the last sum on the second line of \eqref{aux18g} would be replaced by $\sum_{k}|k|^2|\hat \trho_D(k)|^2$, which is finite if and only if $\trho \in H_0^1(\Omega)$. In our case it would be infinity since $\trho|_{\partial \Omega}\not\equiv 0$, so the last inequality of \eqref{aux18g} would fail to hold.
\end{proof}

For our application, we need a time averaged version of \eqref{GNN1}.
\begin{cor}\label{GNNcor}
Let $0 \le s < r < T_*$, and suppose that
\begin{equation}\label{hasstime}
\int_s^r \|\p_{x_1} \rho\|_{\dot H^{-1}_{0}(\Omega)}^2 \,dt  \leq  N^{-1} \int _s^r \|\trho\|_{L^2(\Omega)}^2 \,dt.
\end{equation}
Then there exists a universal constant $C$ such that
\begin{equation}\label{GNNtime1}
\int_s^r \|\trho\|_{L^2(\Omega)}^2 \, dt \le CN^{-1/4} \int_s^r \|\trho\|_{L^1(\Omega)} \|\nabla \trho\|_{L^2(\Omega)}\,dt.
\end{equation}
In particular, \eqref{GNNtime1} implies that the following holds for some universal constant $c>0$:
\begin{equation}\label{GNN}
       \int_{s}^{r}\|\nabla \rho\|_{L^2}^2 \, dt \geq c \rho_M^{-2} N^{1/2}  (r-s)^{-1} \left(\int_{s}^{r}\|\trho\|_{L^2}^2 \, dt\right)^2.
        \end{equation}

\end{cor}
\begin{proof}
Let us denote $S$ the set of times $t$ in $[s,r]$ such that $\|\p_{x_1} \rho (\cdot, t)\|_{\dot H_0^{-1}}^2 \leq 2 N^{-1} \|\tilde{\rho}(\cdot, t)\|_{L^2}^2.$
Then $\int_S \|\tilde{\rho}\|_{L^2}^2 \,dt \geq \frac12 \int_s^r \|\tilde{\rho}\|_{L^2}^2 dt$, since otherwise
\[ \frac12 \int_s^r \|\tilde{\rho}\|_{L^2}^2 < \int_{[s,r] \setminus S} \|\tilde{\rho}\|_{L^2}^2 \,dt \leq \frac{N}{2} \int_{[s,r] \setminus S} \|\p_{x_1} \rho\|_{\dot H^{-1}_{0}(\Omega)}^2 \,dt, \]
a contradiction with \eqref{hasstime}. On the other hand by \eqref{GNN1}, for every $t \in S$ we have $\|\trho\|_{L^2}^2 \le CN^{-1/4} \|\trho\|_{L^1} \|\nabla \trho\|_{L^2}.$
Hence
\begin{align*}
\int_s^r \|\tilde{\rho}\|_{L^2}^2 \,dt \leq 2 \int_S \|\tilde{\rho}\|_{L^2}^2 \,dt \leq 2C N^{-1/4} \int_S \|\trho\|_{L^1} \|\nabla \trho\|_{L^2}\,dt.
\end{align*}
This proves \eqref{GNNtime1}, where we relabeled the constant $2C$ back to $C$.

It remains to prove \eqref{GNN}. Applying Cauchy-Schwarz inequality to \eqref{GNNtime1}, and using $ \|\nabla \rho\|_{L^2}^2 = \|\nabla \trho\|_{L^2}^2 + \|\p_{x_2}\brho\|_{L^2}^2 \geq \|\nabla \trho\|_{L^2}^2
$, we have
\begin{equation}\label{temp00}
 \int_{s}^{r}\|\nabla \rho\|_{L^2}^2 \, dt \geq \int_{s}^{r}\|\nabla \trho\|_{L^2}^2 \, dt  \geq c N^{1/2} \left(\int_{s}^{r}\|\trho\|_{L^2}^2 \, dt\right)^2 \left(\int_{s}^{r} \|\trho\|_{L^1}^2 dt \right)^{-1} .
\end{equation}

Next we claim that $\|\trho\|_{L^1} \lesssim \rho_M$. Using the definition of $\brho$ and the fact that $\brho,\rho \ge 0$ in $\Omega$,
$$
\|\brho\|_{L^1(\Omega)} = \int_\T \int_0^\pi \frac{1}{2\pi}\int_\T \rho(x_1,x_2)dx_1 dx_2 dz = \int_0^\pi\int_\T \rho(x_1,x_2)dx_1 dx_2 = \|\rho_0\|_{L^1}.
$$
Then triangle inequality gives
$
\|\trho\|_{L^1} \le \|\brho\|_{L^1(\Omega)} + \|\rho\|_{L^1} = 2\|\rho_0\|_{L^1} = 4\pi^2\rho_M.
$ Substituting this estimate into \eqref{temp00} implies \eqref{GNN}, thus we finish the proof.
\end{proof}

In the following, recall that $N_0$ is always assumed to be sufficiently large; its size will be determined explicitly in the proof and
will only depend on the initial datum $\rho_0$ and the parameter $g$.

The next lemma relates the diffusion term appearing in \eqref{eq:dtE} to a lower bound for dissipation in $\rho$. More precisely, we have:
\begin{lem}
    \label{lem:diffusion}
    Given any $N \ge N_0$, assume that $2^{N-1} \le \|\rho(\cdot,t) - \rho_M\|_{L^2}^2 \le 2^{N+2}$ for $t \in (s, r)$, where $0 \le s < r < T_*$. Moreover, suppose that
\begin{equation}\label{aux427}
    \left|\int_{s}^{r} \int_\Omega \p_{x_2}\rho dx dt\right| = \Lambda(r - s)
    \end{equation}
    for some $\Lambda >0.$
    Then there exists a universal constant $c>0$ such that
    \begin{equation}\label{aux326b}
    \int_{s}^{r}\|\nabla \rho\|_{L^2}^2 dt \ge c \rho_M^{-1} \Lambda^3(r - s).
    \end{equation}
    \end{lem}
    \begin{proof}
    We first note that the diffusion term can be rewritten as follows:
    \begin{equation}
    \label{eq:diffrewrite}
    \int_{s}^{r}\int_\Omega \p_{x_2}\rho dx dt = \int_{s}^{r}\int_\T \left(\rho(t, x_1, \pi) - \rho(t, x_1,0)\right) dx_1 dt,
    \end{equation}
    where we integrated in $x_2$. In view of \eqref{eq:diffrewrite}, we may without loss of generality assume that $\int_{s}^{r} \int_\Omega \p_{x_2}\rho dx dt \ge 0$, since we may replace $\rho(x_1,\pi)$ by $\rho(x_1,0)$ below to treat the other case with a parallel argument.

    By \eqref{eq:diffrewrite}, \eqref{aux427}, and $\rho \ge 0$, we have
        \begin{equation}
        \label{est:contrassumption}
        \int_{s}^{r}\int_\T \rho(t, x_1, \pi) dx_1 dt \ge \Lambda (r - s).
        \end{equation}
        We also recall the fact that by conservation of mass,
        $$
        \int_{s}^{r} \int_\Omega \rho(t, x_1, x_2) dx dt = \rho_M (r - s).
        $$
        Note that $\int_s^r \|\nabla \rho\|_{L^2}^2 \, dt \geq c 2^N (r-s)$ by Poincar\'e inequality, so we need only consider the situation where $\Lambda$ is sufficiently large, in particular
        $\Lambda > \frac{2\pi}{\rho_M}$ (otherwise \eqref{aux326b} is automatic if $N_0$ is sufficiently large).
        We claim that there exists $y \in [\pi - 2\rho_M\Lambda^{-1}, \pi] \subset [0,\pi]$ such that $\int_{s}^{r} \int_\T \rho(t, x_1, y) dx dt \le \frac{\Lambda}{2}(r - s)$.
        Indeed, suppose it were not the case. Then we must have
        {
        \begin{align*}
            \int_{s}^{r} \int_\Omega \rho(t, x_1, x_2) dx dt &\ge \int_{s}^{r}\int_\T \int_{\pi - 2\rho_M\Lambda^{-1}}^{\pi}\rho(t,x_1,x_2)dx_2dx_1 dt\\
            &> 2\rho_M \Lambda^{-1}\cdot\frac{\Lambda}{2}(r - s) = \rho_M (r - s),
        \end{align*}
        }
        which is a contradiction. Then combining the claim and \eqref{est:contrassumption} yields
        \begin{align*}
            \frac{\Lambda}{2}(r - s) &\le \int_{s}^{r} \int_\T (\rho(t,x_1,\pi) - \rho(t,x_1,y)) dx_1 dt = \int_{s}^{r} \int_\T \int_y^{\pi} \p_{x_2} \rho(t,x_1,x_2)dx_2dx_1 dt\\
            &\le c \rho_M^{1/2} (r - s)^{1/2}\Lambda^{-1/2}\left(\int_{s}^{r}\|\nabla \rho\|_{L^2}^2 dt\right)^{1/2},
        \end{align*}
        where we used H\"older's inequality in the final step. The proof is completed by rearranging the inequality above.
    \end{proof}

The final lemma gives an estimate on the contribution of the Keller-Segel nonlinearity in \eqref{eq:dtE} to the potential energy $E(t).$
\begin{lem}
    \label{lem:ksnonlinear}
    Given any $N \ge N_0$, we assume that $2^{N-1} \le \|\rho(\cdot,t) - \rho_M\|_{L^2}^2 \le 2^{N+2}$ for $t \in (s, r)$, where $0 \le s < r < T_*$. Then there exists a universal constant $C$  such that
    \begin{equation}
        \label{est:ksnonlinear}
        \left|\int_{s}^{r}\int_\Omega \rho\p_{x_2}(-\Delta_N)^{-1}(\rho - \rho_M) dx dt\right| \le C \rho_M^{2/3} 2^{2N/3}(r - s).
    \end{equation}
     \end{lem}
    \begin{proof}
        By standard Sobolev embedding, $L^p$ elliptic estimate for the Neumann problem, and Gagliardo-Nirenberg inequality, we have:
        \begin{align*}
            \left|\int_\Omega \rho \p_{x_2}(-\Delta_N)^{-1}(\rho - \rho_M)\right| &\le \|\rho\|_{L^2}\|\p_{x_2}(-\Delta_N)^{-1}(\rho - \rho_M)\|_{L^2}\\
            &\le C(\|\rho - \rho_M\|_{L^2} + \rho_M)\|\p_{x_2}(-\Delta_N)^{-1}(\rho - \rho_M)\|_{W^{1,6/5}}\\
            &\le C(\|\rho - \rho_M\|_{L^2} + \rho_M)\|\rho - \rho_M\|_{L^{6/5}}\\
            &\le C\rho_M^{2/3}\|\rho - \rho_M\|_{L^2}^{1/3} (\|\rho - \rho_M\|_{L^2} + \rho_M).
        \end{align*}
Here the choice of Sobolev space $W^{1,6/5}$ is only a matter of convenience. Choosing $N_0$ sufficiently large, we have
        $$
        \left|\int_\Omega \rho \p_{x_2}(-\Delta_N)^{-1}(\rho - \rho_M)\right| \le C\rho_M^{2/3}\|\rho - \rho_M\|_{L^2}^{4/3} \le C\rho_M^{2/3}2^{2N/3}.
        $$
        The proof is finished by integrating over $(s, r)$.
    \end{proof}

\subsection{Estimates on the good intervals}
In this section, our ultimate goal is to show the following key estimate:
\begin{prop}
\label{lem:good}
    Assume that $N \ge N_0$, $N_0$ is sufficiently large, and $(t_k, t_{k+1})$ is a good interval of level $N$. Then there exists a universal constant $c_1 >0$  such that
    \begin{equation}
        \label{est:good}
        E(t_{k+1}) - E(t_k) \le -\frac{c_1 g}{N}.
    \end{equation}
\end{prop}
To prove this key proposition, we will first obtain a lower bound of the main term in \eqref{eq:dtE}.
\begin{lem}
\label{lem:damping}
    Let $N_0$ be sufficiently large. Assume that $N \ge N_0$ and $(t_k, t_{k+1})$ is a good interval of level $N.$ Then
    \begin{equation}
        \label{est:damping}
        \int_{t_k}^{t_{k+1}}\|\p_{x_1} \rho\|_{\dot H^{-1}_0}^2 dt \ge N^{-1}\int_{t_k}^{t_{k+1}} \|\rho - \rho_M\|_{L^2}^2 dt > N^{-1}2^{N-1} (t_{k+1} - t_k).
    \end{equation}
    \end{lem}

        \begin{proof}
 Note that it suffices to prove the first inequality of \eqref{est:damping}, since the second inequality directly follows from the fact that $\|\rho(\cdot,t) - \rho_M\|_{L^2}^2 > 2^{N-1}$ on the interval $(t_k, t_{k+1})$.
            To prove the first inequality, suppose on the contrary that
            $$
            \int_{t_k}^{t_{k+1}}\|\p_{x_1} \rho\|_{\dot H^{-1}_0}^2 \,dt < N^{-1}\int_{t_k}^{t_{k+1}} \|\rho - \rho_M\|_{L^2}^2 \,dt.
            $$
            We combine this with \eqref{GNN} and Lemma~\ref{lem_trho}(a) to obtain
            \begin{equation}\label{temp0}
            \int_{t_k}^{t_{k+1}} \|\nabla\rho\|_{L^2}^2 dt \ge c \rho_M^{-2} N^{1/2}2^{2N}(t_{k+1} - t_k).
            \end{equation}
            Integrating the energy estimate \eqref{est:naive} in $(t_k, t_{k+1})$, we have
            \begin{equation}\label{temp1}
            \begin{split}
                \|\rho (\cdot, t_{k+1}) - \rho_M\|_{L^2}^2 - \|\rho (\cdot, t_{k})  - \rho_M\|_{L^2}^2 &\le -\int_{t_k}^{t_{k+1}} \|\nabla \rho\|_{L^2}^2 \,dt + C_1\int_{t_k}^{t_{k+1}} \|\rho - \rho_M\|_{L^2}^4 \,dt\\
                &\le (-c \rho_M^{-2} N^{1/2} + C) 2^{2N}(t_{k+1} - t_k),
                \end{split}
            \end{equation}
            where the second inequality follows from \eqref{temp0} and the fact that $\|\rho(\cdot,t) - \rho_M\|_{L^2}^2 \leq 2^{N+1}$ for $t\in[t_k,t_{k+1}]$.
            Note that the right hand side of \eqref{temp1} would be negative if we choose $N_0$ sufficiently large so that $-c \rho_M^{-2} N_0^{1/2}+ C< 0$. However, by definition of a good interval at level $N$, the left hand side of \eqref{temp1} is equal to $2^N$, yielding a contradiction and thus finishing the proof.
    \end{proof}

The next lemma gives an upper bound of the diffusion term in \eqref{eq:dtE} on a good interval, which will be dominated by the main term when $N$ is sufficiently large.
\begin{lem}
\label{lem:diffgood}
     Let $N_0$ be sufficiently large. Assume that $N \ge N_0$ and $(t_k, t_{k+1})$ is a good interval of level $N$. Then
    \begin{equation}
        \label{est:diffgood}
        \left|\int_{t_k}^{t_{k+1}}\int_\Omega \p_{x_2} \rho dx dt \right| \le C \rho_M^{1/3} 2^{2N/3}(t_{k+1} - t_k),
    \end{equation}
    where $C$ is a universal constant.
\end{lem}
    \begin{proof}
        Let us first write
        $$
        \left|\int_{t_k}^{t_{k+1}}\int_\Omega \p_{x_2} \rho dx dt \right| = \Lambda (t_{k+1} - t_k)
        $$
        for some $\Lambda > 0.$ We then apply Lemma \ref{lem:diffusion} to obtain that
        \begin{equation}\label{aux327a}
        \int_{t_k}^{t_{k+1}}\|\nabla \rho\|_{L^2}^2 dt \ge c \rho_M^{-1} \Lambda^3(t_{k+1} - t_k),
        \end{equation}
        for some $c > 0$. Integrating the naive energy estimate \eqref{est:naive} from $t_k$ to $t_{k+1}$ and using the inequality \eqref{aux327a}, we have
        \begin{align*}
            c \rho_M^{-1} \Lambda^3(t_{k+1} - t_k)&< \|\rho (\cdot, t_{k+1})  - \rho_M\|_{L^2}^2- \|\rho (\cdot, t_{k})  - \rho_M\|_{L^2}^2 + \int_{t_k}^{t_{k+1}} \|\nabla \rho\|_{L^2}^2 dt\\
            &\le C_1\int_{t_k}^{t_{k+1}} \|\rho - \rho_M\|_{L^2}^4 dt \le C2^{2N}(t_{k+1} - t_k),
        \end{align*}
        where we used that $\|\rho (\cdot, t_{k+1})  - \rho_M\|_{L^2}^2 - \|\rho (\cdot, t_{k})  - \rho_M\|_{L^2}^2 = 2^N > 0$ in the first inequality. Rearranging, we conclude that
        $$
        \Lambda \le C \rho_M^{1/3} 2^{2N/3},
        $$
        and the proof is complete.
    \end{proof}

Now we are ready to prove Proposition \ref{lem:good}.
\begin{proof}[Proof of Proposition \ref{lem:good}]
    Invoking Lemmas \ref{lem:ksnonlinear}, \ref{lem:damping} and \ref{lem:diffgood}, we obtain from \eqref{eq:dtE} that:
    \begin{align*}
        E(t_{k+1}) - E(t_k) &\le -g\int_{t_k}^{t_{k+1}} \|\p_{x_1} \rho\|_{\dot H^{-1}_0}^2 dt+ \left|\int_{t_k}^{t_{k+1}}\int_\Omega \p_{x_2} \rho dx dt\right| \\ &\quad + \left|\int_{t_k}^{t_{k+1}}\int_\Omega \rho\p_{x_2}(-\Delta_N)^{-1}(\rho - \rho_M) dxdt\right|\\
        &\le \left(-\frac{g}{2N} + C \rho_M^{1/3} 2^{-N/3} + C \rho_M^{2/3} 2^{-N/3}\right)2^N(t_{k+1} - t_k)
    \end{align*}
    for any $N \ge N_0$. Choosing $N_0$ sufficiently large so that $-\frac{g}{2N_0} + C\rho_M^{1/3}(1 +\rho_M^{1/3})2^{-N_0/3} \le -\frac{g}{4N_0}$, we conclude that
    $$
    E(t_{k+1}) - E(t_k) < -\frac{g}{4N}2^N(t_{k+1} - t_k) \le -\frac{c_1g}{N},
    $$
    by Corollary \ref{timecor}.
    \end{proof}

\subsection{Estimates on the bad intervals}
On the bad intervals, it seems difficult to rule out the situation that the potential energy $E(t)$ may increase. The goal of this section is to show that any growth is of lower order compared to the contributions from good intervals.
As we will see below, the main observation is that any increase in potential energy has to happen in an extremely short time if it were to increase. A precise statement is given as follows.
\begin{prop}
    \label{lem:fastdrop}
     Let $N_0$ be sufficiently large, which only depends on $\rho_M$ and $g$. Assume that $N \ge N_0$ and $(t_k, t_{k+1})$ is a bad interval of level $N$. Then there exists a constant $C_1(\rho_M)$ such that
     \begin{equation}
     \label{est:bad}
     E(t_{k+1}) - E(t_k) \le C_1(\rho_M) 2^{-N/3}.
     \end{equation}
\end{prop}

     \begin{proof}
        First, we estimate the contribution to $E(t)$ from the diffusion term. Write
        \begin{equation}\label{temp2}
        \left|\int_{t_k}^{t_{k+1}}\int_\Omega \p_{x_2} \rho dx dt \right| = \Lambda (t_{k+1} - t_k)
        \end{equation}
        for some $\Lambda > 0$ (if $\Lambda =0$ there is nothing to do for this term).
        Let us apply Lemma \ref{lem:diffusion} and time-integrated version of \eqref{est:naive} to see that
        \begin{align*}
            c \rho_M^{-1} \Lambda^3(t_{k+1} - t_k) &\le \int_{t_k}^{t_{k+1}}\|\nabla \rho\|_{L^2}^2 dt \\ & \le \|\rho (\cdot, t_{k})  - \rho_M\|_{L^2}^2 - \|\rho (\cdot, t_{k+1})  - \rho_M\|_{L^2}^2 + C_1\int_{t_k}^{t_{k+1}} \|\rho- \rho_M\|_{L^2}^4 dt
          \\ & \le 2^N + C_12^{2N+4}(t_{k+1} - t_k).
        \end{align*}
        Taking power $\frac{1}{3}$ on both sides of the above inequality and substituting it into \eqref{temp2}, we have
        \begin{equation}
            \label{est:diffbad2}
            \left|\int_{t_k}^{t_{k+1}}\int_\Omega \p_{x_2} \rho dx dt \right| \le C\rho_M^{1/3}\left(2^{N/3}(t_{k+1} - t_k)^{2/3} + 2^{2N/3}(t_{k+1} - t_k)\right).
        \end{equation}
       Combining \eqref{est:diffbad2} with Lemma \ref{lem:ksnonlinear}, invoking \eqref{eq:dtE}, and choosing $N_0$ sufficiently large, we obtain
        \begin{align}
            \label{est:mainbad1}
            E(t_{k+1}) - E(t_k) &\le -g\int_{t_k}^{t_{k+1}}\|\p_{x_1}\rho\|_{\dot H^{-1}_0}^2 dt + C_0(\rho_M) (2^{N/3}(t_{k+1} - t_k)^{2/3} +2^{2N/3}(t_{k+1} - t_k)),
        \end{align}
        where $C_0(\rho_M) = C(1+\rho_M^{2/3}).$
        Now, we discuss the following dichotomy:

        \medskip

        \noindent\textbf{Case 1.} $\int_{t_k}^{t_{k+1}} \|\trho\|_{L^2}^2 dt < 2^{N-2} (t_{k+1}-t_k)$. In this case, Lemma~\ref{lem_trho}(b) yields that $t_{k+1}-t_k < C \rho_M^4 2^{-2N}$, so setting it into \eqref{est:mainbad1} immediately gives us $  E(t_{k+1}) - E(t_k) \leq C(\rho_M) 2^{-N}.$

        \medskip
        \noindent\textbf{Case 2.} $\int_{t_k}^{t_{k+1}} \|\trho\|_{L^2}^2 dt \geq 2^{N-2} (t_{k+1}-t_k)$. Denoting $A:= 2^{N} (t_{k+1}-t_k)$, we can rewrite \eqref{est:mainbad1} as
        \begin{equation}\label{ineq_A}
        E(t_{k+1}) - E(t_k) \le -g\int_{t_k}^{t_{k+1}}\|\p_{x_1}\rho\|_{\dot H^{-1}_0}^2 dt + C_0(\rho_M) 2^{-N/3} (A^{2/3} + A).
        \end{equation}
        Next we further discuss whether \eqref{hasstime} holds on the interval $(t_k, t_{k+1})$:

        \begin{itemize}
       \item  If \eqref{hasstime} holds, then \eqref{GNN} implies
       \[
              \int_{t_k}^{t_{k+1}}\|\nabla \rho\|_{L^2}^2 \, dt \geq c \rho_M^{-2} N^{1/2}  (t_{k+1}-t_k)^{-1} \left(\int_{t_k}^{t_{k+1}}\|\trho\|_{L^2}^2 \, dt\right)^2 \geq c\rho_M^{-2} N^{1/2} 2^{2N-4} (t_{k+1}-t_k).
	\]
        Integrating \eqref{est:naive} in $(t_k, t_{k+1})$, applying the above inequality and the fact that $\|\rho-\rho_M\|_{L^2}^2< 2^{N+2}$ in $(t_k,t_{k+1})$, we have
        \[
       -2^N= \|\rho-\rho_M\|_{L^2}^2(t_{k+1})- \|\rho-\rho_M\|_{L^2}^2(t_{k})\leq (-c\rho_M^{-2} N^{1/2} + C) 2^{2N} (t_{k+1}-t_k).
        \]
        If we choose $N_0$ large enough such that $c\rho_M^{-2} N_0^{1/2} \geq 2C$, for all $N\geq N_0$ we have
        \[
        t_{k+1}-t_k\leq  C \rho_M^2 N^{-1/2} 2^{-N}.
        \]
        This implies $A = 2^N (t_{k+1}-t_k) \leq  C \rho_M^2 N^{-1/2} $, which can be made less than one by choosing $N_0$ sufficiently large only depending on $\rho_M$. Applying $A\leq 1$ in \eqref{ineq_A} gives the desired inequality.

        \item If \eqref{hasstime} fails, then combining this assumption with the assumption in Case 2 gives
        \[
        \int_{t_k}^{t_{k+1}} \|\p_{x_1} \rho\|_{\dot H^{-1}_{0}}^2 \,dt  \geq  N^{-1}  \int_{t_k}^{t_{k+1}}  \|\trho\|_{L^2}^2 \,dt \geq N^{-1} 2^{N-2} (t_{k+1}-t_k) = \frac{1}{4}N^{-1} A.
        \]
        Putting it into \eqref{est:mainbad1} gives
        \[
         E(t_{k+1}) - E(t_k) \le -\frac{g}{4} N^{-1}A + C_0(\rho_M) 2^{-N/3} (A^{2/3} + A).
        \]
        If $A<1$, then we again have $E(t_{k+1}) - E(t_k) \le 2C_0(\rho_M) 2^{-N/3}.$ If $A>1$, then the inequality can be written as
        \[
        E(t_{k+1}) - E(t_k) \le -\frac{g}{4} N^{-1}A + 2C_0(\rho_M) 2^{-N/3} A = \left( -\frac{g}{4} N^{-1} + 2C_0(\rho_M) 2^{-N/3}\right) A.
        \]
        Since $N^{-1}$ decays slower than $2^{-N/3}$, the right hand side is negative for all $N\geq N_0$ with $N_0$ sufficiently large (which only depends on $g$ and $\rho_M$).
        \end{itemize}
        We have now showed that \eqref{est:bad} holds in all cases, and this finishes the proof.
        \end{proof}

%%%%%%%%%%%%%%%%%%%%%%%%%%%%%%%%%%%%%%%%%%%%%%%%%%%%%%%%%%%%%%%%%%%%%
\section{Proof of the main theorem}
In this section, we prove the global well-posedness of \eqref{eq:KSIPM}.

\begin{proof}[Proof of Theorem \ref{mainthm}]
Let us choose a sufficiently large $N_0\in\mathbb{N}$ for which all previous bounds work,
and in addition
\begin{equation}\label{aux327b}
c_1 g N^{-1} \geq C_1 2^{-N/3}
\end{equation}
for all $N \geq N_0$; here $c_1$ and $C_1$ are constants from the estimates \eqref{est:good} and \eqref{est:bad}
respectively.
 Next, find $Q\geq N_0$ such that
\[ c_1g \sum_{m = N_0}^Q \frac1m > \pi\|\rho_0\|_{L^1}; \]
recall that the right hand side is the maximal possible potential energy due to \eqref{aux13e}, and the potential energy is non-negative for all times.
Let $t_K < T_*$ be the first time when we have $\|\rho(\cdot, t_K) - \rho_M\|_{L^2}^2 = 2^{Q+1}.$
Then
\[ E(t_K)-E(t_1) = \sum_{j=1}^{K-1} (E(t_{j+1})-E(t_j)). \]
Due to Lemma \ref{timeint}, in this sum for each level $N = N_0,\dots, Q,$ the number of good intervals exceeds the number of bad intervals by one.
Then due to \eqref{est:good}, \eqref{est:bad}, and \eqref{aux327b}, we have
\[  E(t_K)-E(t_1) \leq -c_1 g\sum_{m = N_0}^Q \frac1m < -\pi\|\rho_0\|_{L^1}. \]
This implies that $E(t_K) < E(t_1) -\pi\|\rho_0\|_{L^1} < 0$, a contradiction.
\end{proof}
Observe that the above argument can be used to derive a quantitative upper bound on the size of $\|\rho(\cdot, t)- \rho_M\|_{L^2}$ depending on
$g$ and the initial data - but we do not pursue it here.

\noindent {\bf Acknowledgement.} ZH and AK have been partially supported by the NSF-DMS grants 2006372 and 2306726.
 YY has been partially supported by the NUS startup grant A-0008382-00-00, MOE Tier 1 grant A-0008491-00-00, and the Asian Young Scientist Fellowship.
 AK is grateful to Konstantin Kalinin and Felix Otto for a stimulating discussion. 
 The authors would like to thank Siming He for useful discussions, and the anonymous referee for helpful comments.

\end{document}